\documentclass[10 pt, a4 paper, twoside]{article}
\usepackage{amsmath,amssymb,graphicx}
\usepackage{longtable,mathrsfs}
\usepackage{hyperref}
\usepackage[mathscr]{euscript}
\usepackage{amsmath, amsthm}
\usepackage[mathscr]{euscript}
\newtheorem{theorem}{Theorem}

\newtheorem{lemma}{Lemma}

\newtheorem{definition}{Definition}
{ \setlength{\oddsidemargin}{-2mm}
  \setlength{\evensidemargin}{6mm} }
  \newlength{\titleright}
\allowdisplaybreaks
\begin{document}
\begin{center}
{\large\bf On existence of solution to nonlinear $\psi-$Hilfer Cauchy-type problem}

\vskip.20in

Mohammed S Abdo$^{1}$\, S K Panchal$^{2}$ and Sandeep P Bhairat$^{3}$\footnote{author for correspondence: sp.bhairat@marj.ictmumbai.edu.in}\\[2mm]
{\footnotesize
$^{1}$Department of Mathematics, Hodeidah University, Al-Hodeidah-3114, Yemen.\\
$^{2}$Department of Mathematics, Dr Babasaheb Ambedkar Marathwada University, Aurangabad, (M.S.) India.\\
$^{3}$Institute of Chemical Technology Mumbai, Marathwada Campus, Jalna - 431 203 (M.S.) India.\\
[5pt]
}
\end{center}

{\vspace*{0.5pc} \hrule\hrule\vspace*{2pc}}
\hspace{-0.8cm}{\bf Abstract}\\
The aim of this paper is to obtain the existence of unique solution to nonlinear Cauchy-type problem. We consider the implicit nonlinear Cauchy-type problem with $\psi-$Hilfer fractional derivative. The Banach fixed point theorem is used to obtain the existence of a unique solution whereas the generalized Gronwall inequality is used to discuss continuous data dependence of the solution. The results obtained herein are supported with illustrative example.

\hspace{-0.8cm}{\it \footnotesize {\bf Keywords:}} {\small $\psi$-Hilfer derivative, Cauchy-type problem, existence of solution, fixed point theory, generalized Gronwall inequality.}
\thispagestyle{empty}
\section{Introduction}
Fractional Calculus (FC) has glorious history of three decades and has been developing itself in almost all branches of science and engineering.  It has been emerged and spread wings as a new field of applied mathematics research in twenty first century due to applicability in many real world applications, for instance see \cite{HI,FM,RL,SHG,ce1,ce2}. Since the beginning of FC, many fractional differential and integral operators defined and used by timely mathematicians for serving their own purposes.  During the theoretical development of this arbitrary order calculus, Grunwald-Letnikov, Wyel, Riesz, Liouville-Caputo, Riemann-Liouville, Hadamard, Hilfer became more famous in physics, mechanics, material science, signal and image processing, chemical, biological and electrical engineering, economics and mathematical modelling to name few. For details on theory and application of FC, see \cite{TM,FGG,GK,HI,FM,KL1,RL,SHG,ce1,ce2} and references therein.

In 2006, Kilbas et al \cite{KL1} introduced the concept of fractional differentiation of a function with respect to another function in the Riemann-Liouville sense. They further defined suitable weighted spaces and studied some of its properties by using corresponding fractional integral. Using this idea for Caputo fractional derivative, Almaida \cite{AR1} defined fractional derivative of a function with respect to anther function called $\psi$-Caputo derivative. Many researchers used this $\psi$-Caputo fractional derivative and studied some of the qualitative properties of fractional differential equations (FDEs). Recently in 2017, Sousa and Oliviera \cite{VO} proposed interpolator of $\psi$-Riemann-Liouville and $\psi$-Caputo fractional derivatives in Hilfer’s \cite{HI} sense of definition, and named $\psi$-Hilfer fractional derivative. They generalized the Gronwall inequality and discussed the data dependence of Cauchy-type problem in suitable weighted space \cite{JO19}, also see \cite{AP,JK,KAV,JO18}.

Recently, Kuchhe et al \cite{KAV} studied the existence, uniqueness and continuous dependence of solution to Cauchy-type problem using Weissinger fixed point theorem and Picard approximation technique. They further studied linear Cauchy-type problem for $\psi$-Hilfer differential equations with constant as well as variable coefficients in suitable waited space of functions. For some recent works on qualitative properties of Riemann-Liouville, Liouville-Caputo, Hilfer and $\psi$-Hilfer FDEs, see \cite{ABGH,BSP1,MS8,DR,D1,DB,DB4,FK,JK,KL1,KD1,KD2,KAV,DKZ,JN15,SK}.

Motivated by above contributions and the works \cite{KAV,DKZ}, in this paper, we consider the following implicit nonlinear Cauchy-type problem:
\begin{equation}
D_{a^{+}}^{\alpha ,\beta ;\psi }u(t)=f(t,u(t),D_{a^{+}}^{\alpha ,\beta ;\psi
}u(t)),\text{ \ }0<\alpha <1,0\leq \beta \leq 1,t>a,  \label{11}
\end{equation}%
\begin{equation}
I_{a^{+}}^{1-\gamma ;\psi }u(a)=u_{a},\qquad  \qquad \ u_a\in\mathbb{R},\quad\gamma=\alpha +\beta -\alpha \beta ,  \label{12}
\end{equation}%
where $D_{a^{+}}^{\alpha ,\beta ;\psi }$ is the $\psi $-Hilfer fractional derivative \cite{JO18}, $I_{a^{+}}^{1-\gamma ;\psi }$\ is $\psi-$Riemann-Liouville fractional integral, $f:(a,b]\times \mathbb{R}\times \mathbb{R}%
\longrightarrow \mathbb{R}$ is given function
satisfying some assumptions that will be specified in Section \ref{kk} and $%
u_{a}$ is a constant. In this paper, we use Banach fixed point theorem to prove the existence, uniqueness results and
generalized Gronwall inequality to discuss the continuous data dependence of Cauchy-type problem \eqref{11}-\eqref{12}.

The rest of the paper is outlined as follows: In Section \ref{jj}, we list some basic definitions, preliminary facts and lemmas useful
in the subsequent sections. In section \ref{kk}, we prove the equivalence of the $\psi-$Hilfer Cauchy-type problem with Volterra
integral equation. Further, we prove the existence of a uniqueness solution to Cauchy-type problem \ref{11})--(\ref{12}. The continuous dependence on order of differentiation and initial data of considered problem will be discussed in Section \ref{hh}. Finally, an example illustrating our main results will be provided in last section.

\section{Preliminaries}\label{jj}
In this section, we list some notations, basic definitions and preliminary facts which can be found in \cite{KL1,JO18}. Let $[a,b]\subset{\mathbb{R}^+},(0<a<b<\infty),$
and $C[a,b],$ $AC^{n}[a,b],C^{n}[a,b]$ be the spaces of all continuous real functions, $n$-times absolutely continuous functions, $n$-times
continuously differentiable functions on $[a,b]$, respectively. Let $L^{p}(a,b),$ $(1\leq p<\infty )$ be a space of Lebesgue measurable functions defined on $(a,b)$.
Let us define the following norms
\begin{align*}
 \left\Vert h\right\Vert _{L^{p}(a,b)}&=\left[ \int_{a}^{b}\left\vert h(t)\right\vert ^{p}dt\right] ^{\frac{1}{p}}<\infty, \text{ for any } h\in L^{p}(a,b).\\
\left\Vert h\right\Vert _{C[a,b]}&=\max \{\left\vert h(t)\right\vert :t\in \lbrack [a,b]\}, \text{ for any } h\in C[a,b],\\
\text{ and }  &AC^{n}[a,b]=\{h:[a,b]\rightarrow\mathbb{R} | h^{(n-1)}\in AC[a,b]\}.
\end{align*}
We need the following weighted spaces of continuous functions:
\begin{equation*}
C_{\gamma ;\psi }[a,b]=\{h:(a,b]\rightarrow
\mathbb{R}
:(\psi (t)-\psi (a))^{\gamma }h(t)\in C[a,b]\},\text{ \ }0\leq \gamma <1,
\end{equation*}%
\begin{equation*}
C_{\gamma ;\psi }^{n}[a,b]=\{h:(a,b]\rightarrow
\mathbb{R}
:h(t)\in C^{n-1}[a,b];h^{(n)}(t)\in C_{\gamma ;\psi }[a,b]\},\text{ \ }0\leq
\gamma <1,\text{ }n\in
\mathbb{N}
\end{equation*}%
and
\begin{equation*}
C_{\gamma ;\psi }^{\alpha ,\beta }[a,b]=\{h\in C_{\gamma ;\psi }[a,b]:\text{
}D_{a^{+}}^{\alpha ,\beta ;\psi }h\in C_{\gamma ;\psi }[a,b]\},\text{ \ }%
\gamma =\alpha +\beta -\alpha \beta .
\end{equation*}%
In particular, if $n=0$, we have%
\begin{equation*}
C_{\gamma ;\psi }^{0}[a,b]=C_{\gamma ;\psi }[a,b]
\end{equation*}%
with the norms%
\begin{equation*}
\left\Vert h\right\Vert _{C_{\gamma ;\psi }[a,b]}=\left\Vert (\psi (t)-\psi
(a))^{\gamma }h(t)\right\Vert _{C[a,b]}=\max \{\left\vert (\psi (t)-\psi
(a))^{\gamma }h(t)\right\vert :t\in \lbrack a,b]\},
\end{equation*}%
and%
\begin{equation*}
\left\Vert h\right\Vert _{C_{\gamma ;\psi
}^{n}[a,b]}=\sum_{k=0}^{n-1}\left\Vert h^{(k)}\right\Vert
_{C[a,b]}+\left\Vert h^{(n)}\right\Vert _{C_{\gamma ;\psi }[a,b]}.
\end{equation*}
\begin{definition} The familiar Mittag–Leffler functions $E_{\nu}(z)$ and $E_{\nu,\mu}(z)$ are defined by the series: 
\begin{equation}\label{ml}
E_{\nu}(z):=\sum_{k=0}^{\infty}\frac{z^k}{\Gamma(\nu{k}+1)}=:E_{\nu,1}(z) \text{ and }  E_{\nu,\mu}(z):=\sum_{k=0}^{\infty}\frac{z^k}{\Gamma(\nu{k}+\mu)},
\end{equation}
where $\nu,\mu\in\mathbb{C}, Re(\nu)>0$ and $\Gamma(\zeta), \zeta>0,$ is Euler gamma function given by
\begin{equation}\label{gamma}
\Gamma(\zeta)=\int_{0}^{\infty}e^{-t}t^{\zeta-1}dt.
\end{equation}
\end{definition}
\begin{definition}
\label{r} \cite{KL1} The left-sided $\psi $-Riemann-Liouville
fractional integral and fractional derivative of order $\alpha $ $%
(n-1<\alpha <n)$ for an integrable function $h:[a,b]\rightarrow
\mathbb{R}
$ with respect to another function $\psi :[a,b]\rightarrow
\mathbb{R},
$ that is an increasing differentiable function such that $\psi ^{\prime
}(t)\neq 0$, for all $t\in \lbrack a,b],$ $(-\infty \leq a<b\leq +\infty),$
are respectively defined as follows:%
\begin{equation}
I_{a^{+}}^{\alpha ;\psi }h(t)=\frac{1}{\Gamma (\alpha )}\int_{a}^{t}\psi
^{\prime }(s)(\psi (t)-\psi (s))^{\alpha -1}h(s)ds  \label{a2}
\end{equation}%
and%
\begin{eqnarray}
D_{a^{+}}^{\alpha ;\psi }h(t) &=&\left( \frac{1}{\psi ^{\prime }(t)}\frac{d}{%
dt}\right) ^{n}\text{ }I_{a^{+}}^{n-\alpha ;\psi }h(t)  \notag \\
&=&\frac{1}{\Gamma (n-\alpha )}\left( \frac{1}{\psi ^{\prime }(t)}\frac{d}{dt%
}\right) ^{n}\int_{a}^{t}\psi ^{\prime }(s)(\psi (t)-\psi (s))^{n-\alpha
-1}h(s)ds.  \label{a3}
\end{eqnarray}
\end{definition}

\begin{definition}
\cite{AR1} The left-sided $\psi $-Caputo fractional derivative of function
$h\in C^{n}[a,b],$ $(n-1<\alpha <n),$ $n=[\alpha]+1$ with respect to another
function $\psi $ is defined by%
\begin{eqnarray*}
\text{ }^{c}D_{a^{+}}^{\alpha ;\psi }h(t) &=&\text{ }I_{a^{+}}^{n-\alpha
;\psi }\left( \frac{1}{\psi ^{\prime }(t)}\frac{d}{dt}\right) ^{n}h(t) \\
&=&\frac{1}{\Gamma (n-\alpha )}\int_{a}^{t}\psi ^{\prime }(s)(\psi (t)-\psi
(s))^{n-\alpha -1}h_{\psi }^{[n]}(s)ds,
\end{eqnarray*}%
where $h_{\psi }^{[n]}(t)=\left( \frac{1}{\psi ^{\prime }(t)}\frac{d}{dt}%
\right) ^{n}h(t)$ and $\psi $ defined as in Definition \ref{r}.
\end{definition}
Moreover, the $\psi -$Caputo fractional derivative of function $h\in
AC^{n}[a,b]$ is determined as%
\begin{equation*}
^{c}D_{a^{+}}^{\alpha ;\psi }h(t)=D_{a^{+}}^{\alpha ;\psi }\left[
h(t)-\sum_{k=0}^{n-1}\frac{\left[ \frac{1}{\psi ^{\prime }(t)}\frac{d}{dt}%
\right] ^{k}h(a)}{k!}(\psi (t)-\psi (a))^{k}\right] .
\end{equation*}%
\begin{definition}
\label{M1}\cite{JO18} Let $n-1<\alpha <n, \,\, n\in
\mathbb{N}
$, with $[a,b],$ $-\infty \leq a<b\leq +\infty, $ and $%
\psi \in C^{n}([a,b],%
\mathbb{R}
)$ a function such that $\psi (t)$ is increasing and $\psi ^{\prime }(t)\neq
0$, for all $t\in \lbrack a,b].$ The $\psi $-Hilfer fractional derivative
(left-sided) of function $h\in C^{n}([a,b],%
\mathbb{R}
)$ of order $\alpha $ and type $\beta \in [0,1]$ is determined as
\begin{equation*}
D_{a^{+}}^{\alpha ,\beta ;\psi }h(t)=I_{a^{+}}^{\beta (n-\alpha );\psi }
\left[ \frac{1}{\psi ^{\prime }(t)}\frac{d}{dt}\right] ^{n}I_{a^{+}}^{(1-%
\beta )(n-\alpha );\psi }h(t),\text{ }t>a.
\end{equation*}
In other way%
\begin{equation}
D_{a^{+}}^{\alpha ,\beta ;\psi }h(t)=I_{a^{+}}^{\beta (n-\alpha );\psi
}D_{a^{+}}^{\gamma ;\psi }h(t),\text{ }t>a,  \label{z1}
\end{equation}%
where%
\begin{equation*}
D_{a^{+}}^{\gamma ;\psi }h(t)=\left[ \frac{1}{\psi ^{\prime }(t)}\frac{d}{dt}%
\right] ^{n}I_{a^{+}}^{(1-\beta )(n-\alpha );\psi }h(t).
\end{equation*}
In particular, the $\psi $-Hilfer fractional derivative of order $\alpha(0,1)$ and type $\alpha\in[0,1]$, can be written in the following form%
\begin{eqnarray}
D_{a^{+}}^{\alpha ,\beta ;\psi }h(t) &=&\frac{1}{\Gamma (\gamma -\alpha )}%
\int_{a}^{t}(\psi (t)-\psi (s))^{\gamma -\alpha -1}D_{a^{+}}^{\gamma ;\psi
}h(s)ds  \notag \\
&=&I_{a^{+}}^{\gamma -\alpha ;\psi }D_{a^{+}}^{\gamma ;\psi }h(t),
\label{y1}
\end{eqnarray}%
where $\gamma =\alpha +\beta -\alpha \beta $, and $I_{a^{+}}^{\gamma -\alpha
;\psi }(\cdot )$ is defined by \eqref{a2} and $D_{a^{+}}^{\gamma ;\psi
}h(t)=\left[ \frac{1}{\psi ^{\prime }(t)}\frac{d}{dt}\right]
I_{a^{+}}^{1-\gamma ;\psi }h(t).$
\end{definition}

\begin{lemma}
\label{a4} \cite{FK} Let $\alpha >0,$ $0\leq \gamma <1$ and $h\in
L^{1}(a,b)$. Then%
\begin{equation*}
I_{a^{+}}^{\alpha ;\psi }I_{a^{+}}^{\beta ;\psi }h(t)=I_{a^{+}}^{\alpha
+\beta ;\psi }h(t),\text{ }a.e.\text{ }t\in \lbrack a,b].
\end{equation*}%
In particular,
\begin{description}
\item[(i)] if $h\in C_{\gamma ;\psi }[a,b]$. Then $I_{a^{+}}^{\alpha ;\psi
}I_{a^{+}}^{\beta ;\psi }h(t)=I_{a^{+}}^{\alpha +\beta ;\psi }h(t),$ $t\in
(a,b].$

\item[(ii)] If $h\in C[a,b].$ Then $I_{a^{+}}^{\alpha ;\psi }I_{a^{+}}^{\beta
;\psi }h(t)=I_{a^{+}}^{\alpha +\beta ;\psi }h(t),$ $t\in \lbrack a,b].$
\end{description}
\end{lemma}

\begin{lemma}\cite{JO18}
\label{a5} Let $\alpha >0,$ $0\leq \beta \leq 1$\ and $0\leq \gamma <1.$ If $%
h\in C_{\gamma ;\psi }[a,b]$ then%
\begin{equation*}
D_{a^{+}}^{\alpha ,\beta ;\psi }I_{a^{+}}^{\alpha ;\psi }h(t)=h(t),\text{ }%
t\in (a,b].
\end{equation*}
If $h\in C^{1}[a,b]$ then%
\begin{equation*}
D_{a^{+}}^{\alpha ,\beta ;\psi }I_{a^{+}}^{\alpha ;\psi }h(t)=h(t),\text{ }%
t\in \lbrack a,b].
\end{equation*}
\end{lemma}
\begin{lemma}\cite{AP}
\label{r3} Let $0<\alpha <1,$ $0\leq \beta \leq 1$ and $\gamma =\alpha
+\beta -\alpha \beta .$ If $h(t)\in C_{1-\gamma ;\psi }^{\gamma }[a,b]$ then%
\begin{equation}
I_{a^{+}}^{\gamma ;\psi }D_{a^{+}}^{\gamma ;\psi }h(t)=I_{a^{+}}^{\alpha
;\psi }\text{ }D_{a^{+}}^{\alpha ,\beta ;\psi }h(t)  \label{m1}
\end{equation}%
and
\begin{equation}
D_{a^{+}}^{\gamma ;\psi }I_{a^{+}}^{\alpha ;\psi }h(t)=D_{a^{+}}^{\beta
(1-\alpha );\psi }h(t).  \label{m2}
\end{equation}
\end{lemma}
\begin{lemma}\cite{KL1}
\label{r2} Let $t>a,\ \alpha \geq 0,$ and $\delta >0.$ Then
\begin{equation*}
I_{a^{+}}^{\alpha ;\psi }(\psi (t)-\psi (a))^{\delta -1}=\frac{\Gamma
(\delta )}{\Gamma (\delta +\alpha )}(\psi (t)-\psi (a))^{\alpha +\delta -1},
\end{equation*}
and if $0<\alpha <1,$ we have%
\begin{equation*}
D_{a^{+}}^{\alpha ;\psi }(\psi (t)-\psi (a))^{\alpha -1}=0.
\end{equation*}
\end{lemma}

\begin{lemma}
\label{A4} Let $0<\alpha <1$, $0\leq \beta \leq 1,$ $\gamma =\alpha
+\beta -\alpha \beta $ and let $\psi \in C^{1}([a,b],%
\mathbb{R}
)$ be an increasing function such that $\psi (t)\neq 0$, for
all $t\in \lbrack a,b].$ Then

\begin{description}
\item[(i)] $I_{a^{+}}^{\alpha ;\psi }(\cdot )$ maps $C[a,b]$ into $C[a,b].$

\item[(ii)] $I_{a^{+}}^{\alpha ;\psi }(\cdot )$ is bounded from $C_{1-\gamma
;\psi }[a,b]$ into $C_{1-\gamma ;\psi }[a,b].$

\item[(iii)] If $\gamma \leq \alpha ,$ then, $I_{a^{+}}^{\alpha ;\psi
}(\cdot )$ is bounded from $C_{1-\gamma ;\psi }[a,b]$ into $C[a,b]$.
\end{description}
\end{lemma}

\begin{lemma}
\label{z2} Let $\alpha >0$, $0\leq \gamma <1$ and $h\in C_{\gamma ;\psi
}[a,b]$. If $\alpha >\gamma ,\ $then $I_{a^{+}}^{\alpha ;\psi }h$\ $\in
C[a,b] $ and
\begin{equation*}
I_{a^{+}}^{\alpha ;\psi }h(a)=\underset{t\rightarrow a^{+}}{\lim }%
I_{a^{+}}^{\alpha ;\psi }h(t)=0.
\end{equation*}
\end{lemma}

\begin{lemma}
\label{ff} \cite{FK} Let $0\leq \gamma <1,$ $a<c<b,$ $h\in C_{\gamma ;\psi
}[a,c],$ $h\in C[c,b]$ and $h$ is continuous at $c$. Then $h\in C_{\gamma
;\psi }[a,b]. $
\end{lemma}
\begin{theorem}\cite{JO18}
\label{r1} If $h\in C^{n}[a,b],$ $n-1<\alpha <n,$ $0\leq
\beta \leq 1,$ and $\gamma =\alpha +\beta -\alpha \beta .$\ Then for all $t\in (a,b],$
\begin{equation*}
I_{a^{+}}^{\alpha ;\psi }D_{a^{+}}^{\alpha ,\beta ;\psi
}h(t)=h(t)-\sum_{k=1}^{n}\frac{\left[ \psi (t)-\psi (a)\right] ^{\gamma -k}}{%
\Gamma (\gamma -k+1)}h_{\psi }^{(n-k)}I_{a^{+}}^{(1-\beta )(n-\alpha );\psi
}h(a),
\end{equation*}%
In particular, if $0<\alpha <1,$ we have
\begin{equation*}
I_{a^{+}}^{\alpha ;\psi }D_{a^{+}}^{\alpha ,\beta ;\psi }h(t)=h(t)-\frac{%
\left[ \psi (t)-\psi (a)\right] ^{\gamma -1}}{\Gamma (\gamma )}%
I_{a^{+}}^{(1-\beta )(1-\alpha );\psi }h(a).
\end{equation*}

Moreover, if $h\in C_{1-\gamma ;\psi }[a,b]$ and $I_{a^{+}}^{1-\gamma ;\psi
}h\in C_{1-\gamma ;\psi }^{1}[a,b]$ such that $0<\gamma <1.$ Then for all $t\in (a,b],$
\begin{equation*}
I_{a^{+}}^{\gamma ;\psi }D_{a^{+}}^{\gamma ;\psi }h(t)=h(t)-\frac{\left[
\psi (t)-\psi (a)\right] ^{\gamma -1}}{\Gamma (\gamma )}I_{a^{+}}^{1-\gamma
;\psi }h(a).
\end{equation*}%
\end{theorem}

\begin{lemma}
\label{L2} \cite{JO19} (Gronwall's lemma) Let $u,v,$ be two integrable functions and $h$ is
continuous on $[a,b]$. Let $\psi \in C[a,b]$ be an increasing
function such that $\psi ^{\prime }(t)\neq 0,\forall t\in \lbrack a,b].$
Assume that $u$ and $v$ are nonnegative $h$ is nonnegative and
nondecreasing. If
\begin{equation*}
u(t)\leq v(t)+h(t)\int_{a}^{t}\psi ^{\prime }(s)(\psi (t)-\psi (s))^{\alpha
-1}u(s)ds,
\end{equation*}%
then, for all $t\in \lbrack a,b],$we have
\begin{equation}
u(t)\leq v(t)+\int_{a}^{t}\sum_{k=1}^{\infty }\frac{\left[ h(t)\Gamma
(\alpha )\right] ^{k}}{\Gamma (\alpha k)}\psi ^{\prime }(s)(\psi (t)-\psi
(s))^{\alpha k-1}v(s)ds.  \label{t}
\end{equation}%
Further, if $v$ is a nondecreasing function on $[a,b]$ then
\begin{equation*}
u(t)\leq v(t)E_{\alpha }[h(t)\Gamma (\alpha )(\psi (t)-\psi (a))^{\alpha }],
\end{equation*}%
where $E_{\alpha }(\cdot )$ is the Mittag-Leffler function.
\end{lemma}

\begin{theorem}
\label{e} (Banach fixed point theorem) Let $(X,d)$ be a nonempty complete
metric space and let $0\leq L<1$. If $T:X\rightarrow X$ is a mapping such
that for every $x,y\in X$
\begin{equation}
d(T(x),T(y))\leq Ld(x,y),  \label{d1}
\end{equation}%
holds. Then the operator $T$\ has a unique fixed point $x^{\ast }\in
X.$ Furthermore, for any $k\in
\mathbb{N}
,$ if $T^{k}$ is the sequence of operators defined by
\begin{equation*}
T^{1}=T\text{ and }T^{k}=TT^{k-1},\text{ }(k\in
\mathbb{N}
\backslash \{1\}).
\end{equation*}%
then for any $x_{0}\in X,$ $\underset{k\rightarrow \infty }{\lim }%
T^{k}(x_{0})=x^{\ast }$(i.e. $\{T^{k}x_{0}\}_{k=1}^{k=\infty }$ converges to
the above fixed point $x^{\ast }$).\newline
The map $T:X\rightarrow X$ satisfying condition \eqref{d1} is called a
contractive map.
\end{theorem}
We recall the following weighted spaces:
\begin{equation*}
C_{1-\gamma ;\psi }^{\alpha ,\beta }[a,b]=\{h\in C_{1-\gamma ;\psi
}[a,b],D_{a^{+}}^{\alpha ,\beta ;\psi }h\in C_{1-\gamma ;\psi }[a,b]\}
\end{equation*}%
and
\begin{equation}
C_{1-\gamma ;\psi }^{\gamma }[a,b]=\{h\in C_{1-\gamma ;\psi
}[a,b],D_{a^{+}}^{\gamma ;\psi }h\in C_{1-\gamma ;\psi }[a,b]\}.  \label{a1}
\end{equation}%
Since $D_{a^{+}}^{\alpha ,\beta ;\psi }h=I_{a^{+}}^{\beta (1-\alpha );\psi
}D_{a^{+}}^{\gamma ;\psi }h$ clearly by Lemma \ref{A4}, we have $C_{1-\gamma ;\psi
}^{\gamma }[a,b]\subset C_{1-\gamma ;\psi }^{\alpha ,\beta }[a,b].$

\section{Main results}\label{kk}
In this section, we obtain the equivalence between the generalized Cauchy problem \eqref{11}-\eqref{12}
and the Volterra integral equation
\begin{eqnarray}
u(t) &=&\frac{u_{a}}{\Gamma (\gamma )}(\psi (t)-\psi (a))^{\gamma -1}  \notag
\\
&&+\frac{1}{\Gamma (\alpha )}\int_{a}^{t}\psi ^{\prime }(s)(\psi (t)-\psi
(s))^{\alpha -1}f(s,u_{s},D_{a^{+}}^{\alpha ,\beta ;\psi }u_{s})ds.
\label{y}
\end{eqnarray}
Further we prove the existence and uniqueness of solution to \eqref{11}-\ref{12} by means of Banach fixed point
theorem. We need the following hypotheses and auxiliary lemma:
\begin{description}
\item[(A1)] $f:(a,b]\times
\mathbb{R}
\times
\mathbb{R}
\rightarrow
\mathbb{R}
$ be a function such that $f(\cdot ,u,v)\in C_{1-\gamma ;\psi }[a,b]$ for
any $u,v\in C_{1-\gamma ;\psi }[a,b]$.

\item[(A2)] $f(\cdot ,u,v)$ satisfies the Lipschitz's condition with
respect to $u,v$ and is bounded in a region $G\subset
\mathbb{R}
$ such that
\begin{eqnarray}
\left\Vert f(t,u_{1},v_{1})-f(t,u_{2},v_{2})\right\Vert _{C_{1-\gamma
;\psi }[a,b]}  \notag\leq &M\left\Vert u_{1}-u_{2}\right\Vert _{C_{1-\gamma ;\psi
}[a,b]}+M^{\ast }\left\Vert v_{1}-v_{2}\right\Vert _{C_{1-\gamma ;\psi
}[a,b]},  \label{z}
\end{eqnarray}%
for all $t\in (a,b]$ and for all $u_{i},v_{i}\in G,$ $(i=1,2)$ where $M>0$
and $M^{\ast }<1.$
\end{description}
\begin{lemma}
\label{AA} Let $0<\alpha <1,$ $0\leq \beta \leq 1$ and $\gamma =\alpha
+\beta -\alpha \beta ,$ and let $f$ satiesfy (A1). If $u\in C_{1-\gamma
;\psi }^{\gamma }[a,b]$ then $u$ satisfies the cauchy problem (%
\ref{11})--(\ref{12}) if and only if $u$ satisfies the Volterra integral
equation \eqref{y}.
\end{lemma}

\begin{proof}
($\Rightarrow $) Let $u\in C_{1-\gamma ;\psi }^{\gamma }[a,b]$ be a solution
of Cauchy-type problem \eqref{11})--(\eqref{12}. We show that $u$ is also a solution of \eqref{y}.
By the definition of $C_{1-\gamma ;\psi }^{\gamma }[a,b],$\ Lemma \ref{A4}%
(ii), and Definition \ref{M1}, we have%
\begin{equation*}
I_{a^{+}}^{1-\gamma ;\psi }u\in C[a,b]\text{ and }D_{a^{+}}^{\gamma ;\psi }u=%
\left[ \frac{1}{\psi ^{\prime }(t)}\frac{d}{dt}\right] I_{a^{+}}^{1-\gamma
;\psi }u\in C_{1-\gamma ;\psi }[a,b],
\end{equation*}
Since $\psi \in C^{1}[a,b]$ and by definition of $C_{\gamma ;\psi
}^{n}[a,b] $, clearly $I_{a^{+}}^{1-\gamma ;\psi }u\in C_{1-\gamma ;\psi
}^{1}[a,b].$ Hence by using Lemma \ref{r1} and initial condition (\ref{12}), for $t\in (a,b]$
we have%
\begin{eqnarray}
I_{a^{+}}^{\alpha ;\psi }D_{a^{+}}^{\gamma ;\psi }u(t) &=&u(t)-\frac{\left[
\psi (t)-\psi (a)\right] ^{\gamma -1}}{\Gamma (\gamma )}I_{a^{+}}^{1-\gamma
;\psi }u(a)  \notag \\
&=&u(t)-\frac{u_{a}}{\Gamma (\gamma )}\left[ \psi (t)-\psi (a)\right]
^{\gamma -1},  \label{r4}
\end{eqnarray}%
From the fact that, $D_{a^{+}}^{\gamma ;\psi }u\in
C_{1-\gamma ;\psi }[a,b],$ (\ref{y1}), Lemma \ref{r3} and equation \eqref{11}, we
have%
\begin{eqnarray}
I_{a^{+}}^{\gamma ;\psi }D_{a^{+}}^{\gamma ;\psi }u(t) &=&I_{a^{+}}^{\alpha
;\psi }D_{a^{+}}^{\alpha ,\beta ;\psi }u(t)  \notag \\
&=&I_{a^{+}}^{\alpha ;\psi }f(t,u(t),D_{a^{+}}^{\alpha ,\beta ;\psi }u(t))
\label{r5}
\end{eqnarray}
Comparing equations (\ref{r4}) and (\ref{r5}), we obtain the desired integral equation \eqref{y}:
\begin{equation*}
u(t)=\frac{u_{a}}{\Gamma (\gamma )}(\psi (t)-\psi (a))^{\gamma
-1}+I_{a^{+}}^{\alpha ;\psi }f(t,u(t),D_{a^{+}}^{\alpha ,\beta ;\psi }u(t)).
\end{equation*}%

($\Leftarrow $) Assume that $u\in C_{1-\gamma ;\psi }^{\gamma }[a,b]$
satisfies the Volterra integral equation (\ref{y}). We prove that $u$
also satisfies the Cauchy-type problem (\ref{11})--(%
\ref{12}). Applying $D_{a^{+}}^{\gamma
;\psi }$ on both sides of equation \eqref{r2}, in view of Lemmas \ref{r3} and \ref{r2}, we have%
\begin{eqnarray}
D_{a^{+}}^{\gamma ;\psi }u(t) &=&D_{a^{+}}^{\gamma ;\psi }\left[ \frac{u_{a}%
}{\Gamma (\gamma )}(\psi (t)-\psi (a))^{\gamma -1}+I_{a^{+}}^{\alpha ;\psi
}f(t,u(t),D_{a^{+}}^{\alpha ,\beta ;\psi }u(t))\right]  \notag \\
&=&D_{a^{+}}^{\gamma ;\psi }\left[ \frac{u_{a}}{\Gamma (\gamma )}(\psi
(t)-\psi (a))^{\gamma -1}\right] +D_{a^{+}}^{\gamma ;\psi }I_{a^{+}}^{\alpha
;\psi }f(t,u(t),D_{a^{+}}^{\alpha ,\beta ;\psi }u(t))  \notag \\
&=&D_{a^{+}}^{\beta (1-\alpha );\psi }f(t,u(t),D_{a^{+}}^{\alpha ,\beta
;\psi }u(t)).  \label{s8}
\end{eqnarray}
Since $D_{a^{+}}^{\gamma ;\psi }u\in C_{1-\gamma ;\psi }[a,b],$
equation (\ref{s8}) implies that
\begin{equation*}
D_{a^{+}}^{\gamma ;\psi }u(t)=\left[ \frac{1}{\psi ^{\prime }(t)}\frac{d}{dt}%
\right] I_{a^{+}}^{1-\beta (1-\alpha );\psi }f(t,u(t),D_{a^{+}}^{\alpha
,\beta ;\psi }u(t))\in C_{1-\gamma ;\psi }[a,b].
\end{equation*}
As $f(\cdot ,u(\cdot ),D_{a^{+}}^{\alpha ,\beta ;\psi }u(\cdot ))\in
C_{1-\gamma ;\psi }[a,b],$ and from Lemma \ref{A4}(ii), follows
\begin{equation}
I_{a^{+}}^{1-\beta (1-\alpha );\psi }f(\cdot ,u(\cdot ),D_{a^{+}}^{\alpha
,\beta ;\psi }u(\cdot ))\in C_{1-\gamma ;\psi }[a,b].  \label{S7}
\end{equation}%
By the definition of $C_{\gamma ;\psi }^{n}[a,b]$, equation \eqref{S7} means
\begin{equation*}
I_{a^{+}}^{1-\beta (1-\alpha );\psi }f(\cdot ,u(\cdot ),D_{a^{+}}^{\alpha
,\beta ;\psi }u(\cdot ))\in C_{1-\gamma ;\psi }^{1}[a,b].
\end{equation*}
Now apply $I_{a^{+}}^{\beta (1-\alpha );\psi }$ on both sides of equation (\ref{s8}), in view of Lemmas \ref{r1} and \ref{z2}, we
have%
\begin{eqnarray}
&&I_{a^{+}}^{\beta (1-\alpha );\psi }D_{a^{+}}^{\gamma ;\psi }u(t)  \notag \\
&=&I_{a^{+}}^{\beta (1-\alpha );\psi }D_{a^{+}}^{\beta (1-\alpha );\psi
}f(t,u(t),D_{a^{+}}^{\alpha ,\beta ;\psi }u(t))  \notag \\
&=&f(t,u(t),D_{a^{+}}^{\gamma ;\psi }u(t))-\frac{I_{a^{+}}^{1-\beta
(1-\alpha );\psi }f(t,u(t),D_{a^{+}}^{\alpha ,\beta ;\psi }u(t))\mid _{t=a}}{%
\Gamma (\beta (1-\alpha ))}\left[ \psi (t)-\psi (a)\right] ^{\beta (1-\alpha
)-1}  \notag \\
&=&f(t,u(t),D_{a^{+}}^{\alpha ,\beta ;\psi }u(t)).  \label{Z3}
\end{eqnarray}
Comparing equation (\ref{Z3}) and equation (\ref{z1}), we obtain the differential equtions \eqref{11}:
\begin{equation*}
D_{a^{+}}^{\alpha ,\beta ;\psi }u(t)=f(t,u(t),D_{a^{+}}^{\alpha ,\beta ;\psi
}u(t)).
\end{equation*}%

Now we show that $u\in C_{1-\gamma ;\psi }^{\gamma }[a,b]$ given by (\ref{y}) also satisfies the initial condition (\ref{12}). To this
end, multiplying both sides of equation (\ref{y}) by $I_{a^{+}}^{1-\gamma ;\psi
},$ use Lemma \ref{A4}(iii) and Lemma \ref{r2}, we have%
\begin{eqnarray}
I_{a^{+}}^{1-\gamma ;\psi }u(t) &=&I_{a^{+}}^{1-\gamma ;\psi }\left[ \frac{%
u_{a}}{\Gamma (\gamma )}(\psi (t)-\psi (a))^{\gamma -1}+I_{a^{+}}^{\alpha
;\psi }f(t,u(t),D_{a^{+}}^{\alpha ,\beta ;\psi }u(t))\right]  \notag \\
&=&I_{a^{+}}^{1-\gamma ;\psi }\left[ \frac{u_{a}}{\Gamma (\gamma )}(\psi
(t)-\psi (a))^{\gamma -1}\right] +I_{a^{+}}^{1-\gamma ;\psi
}I_{a^{+}}^{\alpha ;\psi }f(t,u(t),D_{a^{+}}^{\alpha ,\beta ;\psi }u(t))
\notag \\
&=&u_{a}+I_{a^{+}}^{1-\beta (1-\alpha );\psi }f(t,u(t),D_{a^{+}}^{\alpha
,\beta ;\psi }u(t))  \label{z4}
\end{eqnarray}
Taking limit $t\rightarrow a$ in equation (\ref{z4}), by Lemma \ref{z2}%
, we conclude that $I_{a^{+}}^{1-\gamma ;\psi }u(a)=u_{a}.$ This
completes the proof.
\end{proof}

Next, we prove the existence and uniqueness of solution for the Cauchy
problem (\ref{11})--(\ref{12}) in the weighted space $C_{1-\gamma ;\psi
}^{\alpha ,\beta }[a,b]$ by using the Banach fixed point theorem.

\begin{theorem}
\label{BB} Let $0<\alpha <1,$ $0\leq \beta \leq 1$ and $\gamma =\alpha
+\beta -\alpha \beta .$ Assume that the hypotheses (A1) and (A2) are
fulfilled. Then there exists a unique solution $u$ for the Cauchy type
problem (\ref{11})--(\ref{12}) in the space $C_{1-\gamma ;\psi }^{\alpha
,\beta }[a,b].$
\end{theorem}
\begin{proof}
Let us rewrite Cauchy-type problem \eqref{11}-\eqref{12} as:
\begin{equation*}
D_{a^{+}}^{\alpha ,\beta ;\psi }u(t)=F_{u}(t),\text{ \ }I_{a^{+}}^{1-\gamma
;\psi }u(a)=u_{a},\text{ \ }\gamma =\alpha +\beta -\alpha \beta ,
\end{equation*}%
then by Lemma \ref{AA},%
\begin{equation}
u(t)=u_{0}(t)+\frac{1}{\Gamma (\alpha )}\int_{a}^{t}\psi ^{\prime }(s)(\psi
(t)-\psi (s))^{\alpha -1}F_{u}(s)ds,\ \ t>a,
\label{A22}
\end{equation}%
where
\begin{equation}
u_{0}(t)=\frac{u_{a}}{\Gamma (\gamma )}(\psi (t)-\psi (a))^{\gamma -1}
\label{A2}
\end{equation}%
and
\begin{equation}
F_{u}(t)=f(t,u_{0}(t)+I_{a^{+}}^{\alpha ;\psi }F_{u}(t),F_{u}(t)).
\label{D3}
\end{equation}%
That is
\begin{equation}
u(t)=u_{0}(t)+I_{a^{+}}^{\alpha ;\psi }F_{u}(t).  \label{D4}
\end{equation}
To start with, we partition the interval $(a,b]$ into $N$
subintervals namely, $(t_{0},t_{1}],[t_{1},t_{2}],...[t_{N-1},t_{N}]$ on which the operator $T$ is a contraction mapping
on all subintervals, where $a=t_{0}<t_{1}<...<t_{N}=b$.

Let $C_{1-\gamma ;\psi }[a,b]$ is a complete metric space with the metric
\begin{equation*}
d(u_{1},u_{2})=\left\Vert u_{1}-u_{2}\right\Vert _{C_{1-\gamma ;\psi }[a,b]}=%
\underset{t\in \lbrack a,b]}{\max }\left[ (\psi (t)-\psi (a)\right]
^{1-\gamma }\left\vert u_{1}(t)-u_{2}(t)\right\vert
\end{equation*}%
for $u_{1},u_{2}\in C_{1-\gamma ;\psi }[a,b]$ and consider the operator $%
T:C_{1-\gamma ;\psi }[a,b]\rightarrow C_{1-\gamma ;\psi }[a,b]$ defined by%
\begin{equation}
(Tu)(t)=u_{0}(t)+\frac{1}{\Gamma (\alpha )}\int_{a}^{t}\psi ^{\prime
}(s)(\psi (t)-\psi (s))^{\alpha -1}F_{u}(s)ds,\ \ t\in \lbrack a,b].
\label{A3}
\end{equation}%
Now, on the interval $(a,t_{1}]$ note that $\ C_{1-\gamma ;\psi }([a,t_{1}]$
is a complete metric space with the metric%
\begin{equation*}
d(u_{1},u_{2})=\left\Vert u_{1}-u_{2}\right\Vert _{C_{1-\gamma ;\psi
}[a,t_{1}]}=\underset{t\in \lbrack a,t_{1}]}{\max }\left[ (\psi (t)-\psi (a)%
\right] ^{1-\gamma }\left\vert u_{1}(t)-u_{2}(t)\right\vert ,
\end{equation*}%
for all $u_{1},u_{2}\in C_{1-\gamma ;\psi }[a,t_{1}].$ Choose $t_{1}\in
(a,b] $ such that%
\begin{equation}
\frac{\Gamma (\gamma )\left[ (\psi (t_{1})-\psi (a)\right] ^{\alpha }}{%
\Gamma (\gamma +\alpha )}\frac{M}{1-M^{\ast }}<1  \label{t1}
\end{equation}%
holds, where $M$ and $M^{\ast }$ are defined as in the condition (A2).

Next, we show that $Tu(t)\in C_{1-\gamma ;\psi }[a,t_{1}]$ for all $t\in
(a,t_{1}].$ From equation (\ref{A2}) it follows that $u_{0}(t)\in C_{1-\gamma
;\psi }[a,t_{1}]$ since $\left[ (\psi (t)-\psi (a)\right] ^{1-\gamma
}u_{0}(t)=\frac{u_{a}}{\Gamma (\gamma )}\in C[a,t_{1}]$.\ In view of Lemma %
\ref{A4}(ii), then $I^{\alpha ;\psi }f$ is bounded from $C_{1-\gamma ;\psi
}[a,t_{1}]$ into $C_{1-\gamma ;\psi }[a,t_{1}]$ which implies $Tu\in
C_{1-\gamma ;\psi }[a,t_{1}],$ i.e. $T:C_{1-\gamma ;\psi
}[a,t_{1}]\rightarrow C_{1-\gamma ;\psi }[a,t_{1}]$.

Now, we prove that $T$ has a fixed point in $C_{1-\gamma ;\psi }[a,t_{1}]$
which is the unique solution to (\ref{11})-(\ref{12}) on $(a,t_{1}].$ To
this end, it is enough to show that the operator $T$ is a contraction map.
Indeed, from equation (\ref{A3}), assumption (A2), Lemma's \ref{r2}, \ref{A4}, and
for any $u_{1},u_{2}\in C_{1-\gamma ;\psi }[a,t_{1}],$ we have%
\begin{align}
|{[\psi (t)-\psi (a)]}^{1-\gamma }(Tu_{1})(t)&-{[\psi (t)-\psi (a)]}^{1-\gamma}(Tu_{2})(t)| \notag \\
&=\left\vert \left[ (\psi (t)-\psi (a)\right] ^{1-\gamma }I_{a^{+}}^{\alpha
;\psi }(F_{u_{1}}(t)-F_{u_{2}}(t))\right\vert  \notag \\
&=\left[ (\psi (t)-\psi (a)\right] ^{1-\gamma }\left\vert I_{a^{+}}^{\alpha
;\psi }\left[ (\psi (t)-\psi (a)\right] ^{\gamma -1}\right\vert  \notag \\
&\times \left[ (\psi (t)-\psi (a)\right] ^{1-\gamma }\left\vert
F_{u_{1}}(t)-F_{u_{2}}(t)\right\vert  \notag \\
&\leq \left\Vert I_{a^{+}}^{\alpha ;\psi }\left[ (\psi (t)-\psi (a)\right]
^{\gamma -1}\right\Vert _{C_{1-\gamma ;\psi }[a,t_{1}]}\left\Vert
F_{u_{1}}(t)-F_{u_{2}}(t)\right\Vert _{C_{1-\gamma ;\psi }[a,t_{1}]}  \notag
\\
&\leq \frac{\Gamma (\gamma )\left[ (\psi (t_{1})-\psi (a)\right] ^{\alpha }%
}{\Gamma (\gamma +\alpha )}\frac{\left[ (\psi (t_{1})-\psi (a)\right]
^{\gamma }}{(\psi (t_{1})-\psi (a)}\left\Vert
F_{u_{1}}(t)-F_{u_{2}}(t)\right\Vert _{C_{1-\gamma ;\psi }[a,t_{1}]} \label{d2}
\end{align}%
and
\begin{align}\label{D5}
{\| F_{u_{1}}(t)-F_{u_{2}}(t)\|}_{C_{1-\gamma;\psi}[a,t_{1}]}=&{\|f(t,u_{1}(t),F_{u_{1}}(t))-f(t,u_{2}(t),F_{u_{2}}(t))\|}_{C_{1-\gamma;\psi}[a,t_{1}]} \nonumber\\
&\leq M{\| u_{1}-u_{2}\|}_{C_{1-\gamma ;\psi}[a,t_{1}]}+M^{\ast}{F_{u_{1}}(t)-F_{u_{2}}(t)\|}_{C_{1-\gamma ;\psi}[a,t_{1}]}\nonumber\\
&\leq \frac{M}{1-M^{\ast}}{\|u_{1}-u_{2}\|}_{C_{1-\gamma ;\psi}[a,t_{1}]}.
\end{align}
Further by \eqref{d2} and \eqref{D5}, we get%
\begin{eqnarray*}
\left\Vert Tu_{1}-Tu_{2}\right\Vert _{_{C_{1-\gamma ;\psi }[a,t_{1}]}}
\leq \frac{\Gamma (\gamma )\left[ (\psi (t_{1})-\psi (a)\right] ^{\alpha }%
}{\Gamma (\gamma +\alpha )}\frac{\left[ (\psi (t_{1})-\psi (a)\right]
^{\gamma }}{(\psi (t_{1})-\psi (a)}\frac{M}{1-M^{\ast }}\left\Vert
u_{1}-u_{2}\right\Vert _{C_{1-\gamma ;\psi }[a,t_{1}]}.
\end{eqnarray*}

Since $\frac{\left[ (\psi (t_{1})-\psi (a)\right] ^{\gamma }}{(\psi
(t_{1})-\psi (a)}<1$ and by condition \eqref{t1}, we conclude that $T$
is a contraction map. By the Banach fixed point theorem, we can deduce that $T$
has a fixed point $u^{\ast }\in C_{1-\gamma ;\psi }[a,t_{1}]$, which is just
the unique solution to the integral equation \eqref{A22} on $(a,t_{1}]$.
This solution $u^{\ast }$ is obtained as a limit of a convergent sequence $
(T^{k}u_{0}^{\ast })_{k\in
\mathbb{N}
}$ as follows:
\begin{equation}
\left\Vert T^{k}u_{0}^{\ast }-u^{\ast }\right\Vert _{_{C_{1-\gamma ;\psi
}[a,t_{1}]}}\rightarrow 0\text{ as }k\rightarrow \infty.  \label{z6}
\end{equation}%
For any function $u_{0}^{\ast }\in C_{1-\gamma ;\psi }[a,t_{1}]$ and%
\begin{equation}
T^{k}u_{0}^{\ast }(t)=TT^{k-1}u_{0}^{\ast }(t)=u_{0}(t)+I_{a^{+}}^{\alpha
;\psi }F_{T^{k-1}u_{0}^{\ast }}(t),\text{ }k\in
\mathbb{N}
.  \label{z7}
\end{equation}
Set $u_{0}^{\ast }(t):=u_{0}(t)$ such that $u_{0}(t)$ is defined by \eqref{A2} and we consider%
\begin{equation}
u_{k}(t)=T^{k}u_{0}^{\ast }(t),\text{ }t\in (a,t_{1}],\text{ }k\in
\mathbb{N}
,  \label{z5}
\end{equation}%
then equation \eqref{z5} leads to%
\begin{equation}
u_{k}(t)=u_{0}(t)+I_{a^{+}}^{\alpha ;\psi }F_{u_{k-1}}(t),\text{ }k\in
\mathbb{N}
.  \label{z8}
\end{equation}
On the other hand, by equations \ref{z7}, \eqref{z5} and \eqref{z8}, limit in \eqref{z6} can be rewritten as%
\begin{equation*}
\left\Vert u_{k}-u^{\ast }\right\Vert _{_{C_{1-\gamma ;\psi
}[a,t_{1}]}}\rightarrow 0,\text{ as }k\rightarrow \infty .
\end{equation*}

If $t_{1}\neq b$ then we consider the interval $[t_{1},b].\ $For $t\in
\lbrack t_{1},b],$ we consider $u\in C[t_{1},b],$ and rewrite integral
equation \eqref{A22} as
\begin{eqnarray}
u(t) &=&Tu(t)=u_{01}(t)+\frac{1}{\Gamma (\alpha )}\int_{t_{1}}^{t}\psi
^{\prime }(s)(\psi (t)-\psi (s))^{\alpha -1}F_{u}(s)ds  \notag \\
&=&u_{01}(t)+I_{t_{1}^{+}}^{\alpha ;\psi }F_{u}(t),  \label{A1}
\end{eqnarray}%
where%
\begin{eqnarray}
u_{01}(t) &=&u_{0}(t)+\frac{1}{\Gamma (\alpha )}\int_{a}^{t_{1}}\psi
^{\prime }(s)(\psi (t)-\psi (s))^{\alpha -1}F_{u}(s)ds  \notag \\
&=&u_{0}(t)+I_{a^{+}}^{\alpha ;\psi }F_{u}(t).  \label{e1}
\end{eqnarray}
Let us choose $t_{2}\in (t_{1},b]$ such that%
\begin{equation}
\frac{\Gamma (\gamma )\left[ (\psi (t_{2})-\psi (t_{1})\right] ^{\alpha }}{%
\Gamma (\gamma +\alpha )}\frac{M}{1-M^{\ast }}<1.  \label{t2}
\end{equation}
As seen above, $u_{01}(t)$ is uniquely defined on $[a,t_{1}],$ and the
integral in equation \eqref{e1} can be considered as a known function i.e. $%
u_{01}(t)\in C_{1-\gamma ;\psi }[t_{1},t_{2}]$. By equation \eqref{A1}, assumption
(A2), Lemma \ref{A4}(i) and relation \eqref{D5}, for $u_{1},u_{2}\in
C[t_{1},t_{2}]$ and $t\in \lbrack t_{1},t_{2}],$ we have%
\begin{align*}
\left\vert (Tu_{1})(t)-(Tu_{2})(t)\right\vert =&\left\vert I_{t_{1}^{+}}^{\alpha ;\psi
}(F_{u_{1}}(t)-F_{u_{2}}(t))\right\vert \\
&\leq \left\Vert I_{t_{1}^{+}}^{\alpha ;\psi
}(F_{u_{1}}(t)-F_{u_{2}}(t))\right\Vert _{C[t_{1},t_{2}]} \\
&\leq \left\Vert I_{t_{1}^{+}}^{\alpha ;\psi }\left\vert
F_{u_{1}}(t)-F_{u_{2}}(t)\right\vert \right\Vert _{C[t_{1},t_{2}]} \\
&\leq \frac{\Gamma (\gamma )\left[ (\psi (t_{2})-\psi (t_{1})\right]
^{\alpha }}{\Gamma (\gamma +\alpha )}\left\Vert
F_{u_{1}}(t)-F_{u_{2}}(t)\right\Vert _{C[t_{1},t_{2}]} \\
&\leq \frac{\Gamma (\gamma )\left[ (\psi (t_{2})-\psi (t_{1})\right]
^{\alpha }}{\Gamma (\gamma +\alpha )}\frac{M}{1-M^{\ast }}\left\Vert
u_{1}-u_{2}\right\Vert _{C[t_{1},t_{2}]}.
\end{align*}
From condition \eqref{t2}, we obtain
\begin{equation*}
\left\Vert Tu_{1}-Tu_{2}\right\Vert _{C[t_{1},t_{2}]}\leq \eta \left\Vert
u_{1}-u_{2}\right\Vert _{C[t_{1},t_{2}]},
\end{equation*}%
where $\eta =\frac{\Gamma (\gamma )\left[ (\psi (t_{2})-\psi (t_{1})\right]
^{\alpha }}{\Gamma (\gamma +\alpha )}\frac{M}{1-M^{\ast }}<1.$ This means
that $T$ is a contraction. Since $F_{u}(t)\in C[t_{1},t_{2}]$ for any
function $u\in C[t_{1},t_{2}],$ then Lemma \ref{A4}(i) yields that $%
I_{t_{1}^{+}}^{\alpha ;\psi }F_{u}(t)\in C[t_{1},t_{2}].$ Thus, $Tu\in
C[t_{1},t_{2}].$ Consequently, the operator $T$ maps $C[t_{1},t_{2}]$ into $%
C[t_{1},t_{2}]$. An application of Banach fixed point theorem shows that
there exists a unique solution $u^{\ast }\in C[t_{1},t_{2}]$ to equation \eqref{A22} on the interval $[t_{1},t_{2}]$. This solution $u^{\ast }$ which is
defined explicitly as a limit of iterations of the mapping $T$. i.e.%
\begin{equation}
\left\Vert T^{k}u_{1}^{\ast }-u^{\ast }\right\Vert
_{_{C[t_{1},t_{2}]}}\rightarrow 0,\text{ as }k\rightarrow \infty .
\label{zz6}
\end{equation}%
Therefore, for any function $u_{1}^{\ast }\in C[t_{1},t_{2}]$ we have%
\begin{eqnarray}
T^{k}u_{1}^{\ast }(t) &=&TT^{k-1}u_{1}^{\ast }(t)  \notag \\
&=&u_{01}(t)+I_{t_{1}^{+}}^{\alpha ;\psi }F_{T^{k-1}u_{1}^{\ast }}(t),\ \
t>t_{1},\text{ }k\in
\mathbb{N}
.  \label{zz7}
\end{eqnarray}
Set $u_{1}^{\ast }(t):=u_{01}(t)$ such that $u_{01}(t)$ is defined by equation \eqref{e1} and we consider%
\begin{equation}
u_{k}(t)=T^{k}u_{1}^{\ast }(t),\text{ }t\in \lbrack t_{1},t_{2}],\text{ }%
k\in
\mathbb{N}
,  \label{zz5}
\end{equation}%
then equation \eqref{zz5} leads to%
\begin{equation}
u_{k}(t)=u_{01}(t)+I_{t_{1}^{+}}^{\alpha ;\psi }F_{u_{k-1}}(t),\ \ t>t_{1},%
\text{ }k\in
\mathbb{N}
.  \label{zz8}
\end{equation}
On the other hand, in the light of \eqref{zz7}, \eqref{zz5} and \eqref{zz8},
equation \eqref{zz6} can be rewritten as%
\begin{equation*}
\left\Vert u_{k}-u^{\ast }\right\Vert _{C[t_{1},t_{2}]}\rightarrow 0,\text{
as }k\rightarrow \infty .
\end{equation*}
Since $u_{0}^{\ast }\in C_{1-\gamma ;\psi }[a,t_{1}]$ and $%
u_{1}^{\ast }\in C[t_{1},t_{2}]$ implies $u_{0}^{\ast }(t_{1})=u_{1}^{\ast
}(t_{1})$. Therefore, if%
\begin{equation*}
u^{\ast }(t)=\left\{
\begin{array}{c}
u_{0}^{\ast }(t),\qquad a<t\leq t_{1}, \\
u_{1}^{\ast }(t),\qquad t_{1}\leq t\leq t_{2},%
\end{array}%
\right.
\end{equation*}%
then by Lemma \ref{ff}, $u^{\ast }\in C_{1-\gamma ;\psi }[a,t_{2}].$ So $%
u^{\ast }$ is the unique solution of equation \eqref{A22} in $C_{1-\gamma ;\psi
}[a,t_{2}]$ on the interval $(a,t_{2}].$

If $t_{2}\neq b,$ we repeat the process $N-2$ times to obtain the unique solution $u_{k}^{\ast }\in
C_{1-\gamma ;\psi }[t_{k},t_{k+1}]$\ for equation \eqref{A22} on the interval $%
[t_{k},t_{k+1}],$ where $a=t_{0}<t_{1}<....<t_{N}=b$ such that
\begin{equation*}
\frac{\Gamma (\gamma )\left[ (\psi (t_{k+1})-\psi (t_{k})\right] ^{\alpha }}{%
\Gamma (\gamma +\alpha )}\frac{M}{1-M^{\ast }}<1.
\end{equation*}
Consequently, we have the unique solution $u^{\ast }\in C_{1-\gamma ;\psi
}[a,b]$ of integral equation (\ref{A22}) given by%
\begin{equation*}
u^{\ast }(t)=u_{k}^{\ast }(t),\qquad t\in (t_{k},t_{k+1}],\text{ }%
k=0,1,...N-1.
\end{equation*}

Finally, we must show that the obtained unique solution $u^{\ast }\in C_{1-\gamma
;\psi }^{\alpha ,\beta }[a,b].$ So, we need to show that $D_{a^{+}}^{\alpha
,\beta ;\psi }u^{\ast }(\cdot )\in C_{1-\gamma ;\psi }[a,b].$ We know that $%
u^{\ast }\in C_{1-\gamma ;\psi }[a,b]$\ is the limit of the sequence $%
u_{k}\in C_{1-\gamma ;\psi }[a,b],$ $k\in
\mathbb{N}
,$ i.e.
\begin{equation}
\left\Vert u_{k}-u^{\ast }\right\Vert _{_{C_{1-\gamma ;\psi
}[a,b]}}\rightarrow 0,\text{ as }k\rightarrow \infty .  \label{s9}
\end{equation}
Therefore, we use fractional differential equation \eqref{11}, Lipschitz
condition (A2) and equation \eqref{s9} to obtain%
\begin{align*}
\| D_{a^{+}}^{\alpha ,\beta ;\psi }u_{k}-D_{a^{+}}^{\alpha,\beta;\psi }u^{\ast }\|_{_{C_{1-\gamma ;\psi }[a,b]}}&=\|f(t,u_{k},F_{u_{k}}(t))-f(t,u^{\ast },F_{u^{\ast}}(t))\|_{_{C_{1-\gamma ;\psi }[a,b]}} \\
&\leq M\|u_{k}-u^{\ast }\|_{C_{1-\gamma ;\psi}[a,b]}+M^{\ast }\|F_{u_{k}}(t)-F_{u^{\ast }}(t)\|_{C_{1-\gamma ;\psi }[a,b]} \\
&\leq [ M+\frac{M}{1-M^{\ast }}] \|u_{k}-u^{\ast}\|_{C_{1-\gamma ;\psi }[a,b]}.
\end{align*}

Taking limit as $k\rightarrow \infty,$ the right-hand side of the
above inequality tends to zero independently. i.e.%
\begin{equation*}
\left\Vert D_{a^{+}}^{\alpha ,\beta ;\psi }u_{k}-D_{a^{+}}^{\alpha ,\beta
;\psi }u^{\ast }\right\Vert _{_{C_{1-\gamma ;\psi }[a,b]}}\rightarrow 0,%
\text{ as }k\rightarrow \infty .
\end{equation*}%
This last expression gives us to $D_{a^{+}}^{\alpha ,\beta ;\psi }u^{\ast
}\in C_{1-\gamma ;\psi }[a,b],$ if $D_{a^{+}}^{\alpha ,\beta ;\psi }u_{k}\in
C_{1-\gamma ;\psi }[a,b]$ $k\in
\mathbb{N}
.$

Since $D_{a^{+}}^{\alpha ,\beta ;\psi
}u_{k}(t)=f(t,u_{k-1}(t),F_{u_{k-1}}(t)),$ then by the preceding argument $f(\cdot ,u^{\ast }(\cdot ),F_{u^{\ast }}(\cdot ))\in
C_{1-\gamma ;\psi }[a,b]$ for any $u^{\ast }\in C_{1-\gamma ;\psi }[a,b].$
Consequently, $u^{\ast }\in C_{1-\gamma ;\psi }^{\alpha ,\beta }[a,b]$ and the proof is complete.
\end{proof}

\section{Continuous dependence} \label{hh}
In this section, we discuss the continuous dependence of solution of an
implicit fractional differential equation involving Hilfer derivative with
respect to another function via the generalized Gronwall inequality.

\begin{theorem}
Let $\psi \in C([a,b],%
\mathbb{R}
)$ a function such that it is increasing and $\psi ^{\prime }(t)\neq 0$, for
all $t\in \lbrack a,b]$, and $f\in C([a,b],%
\mathbb{R}
)$ satisfying Lipschitz condition \eqref{z} on $%
\mathbb{R}
$. Let \ $\alpha >0,$ $\epsilon >0$ such that $0<\alpha -\epsilon <\alpha
\leq 1$ and $0\leq \beta \leq 1$. For any $a\leq t\leq b$, assume that $u$ is
the solution of the Cauchy-type problem (\ref{11})-(\ref{12}) and $u^{\ast }$ is the
solution of the following problem:%
\begin{equation}
D_{a^{+}}^{\alpha -\epsilon ,\beta ;\psi }u^{\ast }(t)=f(t,u^{\ast
}(t),D_{a^{+}}^{\alpha -\epsilon ,\beta ;\psi }u^{\ast }(t)),\text{ \ }%
0<\alpha <1,0\leq \beta \leq 1,t>a,  \label{q1}
\end{equation}%
\begin{equation}
I_{a^{+}}^{1-\gamma -\epsilon (\beta -1);\psi }u^{\ast }(a)=u_{a}^{\ast
}.\qquad \qquad \ \ \qquad \ \gamma =\alpha +\beta -\alpha \beta .
\label{q2}
\end{equation}%
Then, for $a<t\leq b,$%
\begin{align*}
\left\vert u^{\ast }(t)-u(t)\right\vert \leq A(t)+\int_{a}^{t}\left[ \sum_{k=1}^{\infty }\left( \frac{M\Gamma
(\alpha -\epsilon )}{\Gamma (\alpha )(1-M^{\ast })}\right) ^{k}\frac{\psi
^{\prime }(s)(\psi (t)-\psi (s))^{k(\alpha -\epsilon )-1}}{\Gamma (k(\alpha
-\epsilon ))}A(s)\right] ds,
\end{align*}%
where%
\begin{align}\label{t4}
A(t)=&\left\vert \frac{u_{a}^{\ast }(\psi (t)-\psi (a))^{\gamma +\epsilon
(\beta -1)-1}}{\Gamma (\gamma +\epsilon (\beta -1))}-\frac{u_{a}(\psi
(t)-\psi (a))^{\gamma -1}}{\Gamma (\gamma )}\right\vert \\ \nonumber
&+\left\Vert f\right\Vert \left\vert \frac{(\psi (t)-\psi (a))^{\alpha
-\epsilon }}{\Gamma (\alpha -\epsilon +1)}-\frac{(\psi (t)-\psi (a))^{\alpha
-\epsilon }}{\Gamma (\alpha -\epsilon )\Gamma (\alpha )}\right\vert
+\left\Vert f\right\Vert \left\vert \frac{(\psi (t)-\psi (a))^{\alpha
-\epsilon }}{\Gamma (\alpha -\epsilon )\Gamma (\alpha )}-\frac{(\psi
(t)-\psi (a))^{\alpha }}{\Gamma (\alpha +1)}\right\vert,
\end{align}
$\left\Vert f\right\Vert =\underset{t\in \lbrack a,b]}{\max }\left\vert
f(t,u(t),F_{u}(t))\right\vert $ and $F_{u}(t)=D_{a^{+}}^{\alpha ,\beta ;\psi
}u(t).$
\end{theorem}
\begin{proof}
The Cauchy-type problems \eqref{11}-\eqref{12} and \eqref{q1}-\eqref{q2}, have
integral equations which are given by
\begin{eqnarray*}
u(t) &=&\frac{u_{a}}{\Gamma (\gamma )}(\psi (t)-\psi (a))^{\gamma -1} \\
&&+\frac{1}{\Gamma (\alpha )}\int_{a}^{t}\psi ^{\prime }(s)(\psi (t)-\psi
(s))^{\alpha -1}f(s,u(s),F_{u}(s))ds,\text{ }t>a,
\end{eqnarray*}%
and
\begin{eqnarray*}
u^{\ast }(t) &=&\frac{u_{a}^{\ast }}{\Gamma (\gamma +\epsilon (\beta -1))}%
(\psi (t)-\psi (a))^{\gamma +\epsilon (\beta -1)-1} \\
&&+\frac{1}{\Gamma (\alpha -\epsilon )}\int_{a}^{t}\psi ^{\prime }(s)(\psi
(t)-\psi (s))^{\alpha -\epsilon -1}f(s,u^{\ast }(s),F_{u^{\ast }}(s))ds,%
\text{ }t>a,
\end{eqnarray*}%
respectively. It follows that%
\begin{align*}
|u^{\ast}(t)-u(t)| \leq&\bigg|\frac{u_{a}^{\ast }(\psi (t)-\psi (a))^{\gamma +\epsilon(\beta -1)-1}}{\Gamma (\gamma +\epsilon (\beta -1))}-\frac{u_{a}(\psi(t)-\psi (a))^{\gamma -1}}{\Gamma (\gamma )}\bigg| \\
&+\bigg| \frac{1}{\Gamma (\alpha -\epsilon )}\int_{a}^{t}\psi ^{\prime}(s)(\psi (t)-\psi (s))^{\alpha -\epsilon -1}f(s,u^{\ast }(s),F_{u^{\ast}}(s))ds\\
&-\frac{1}{\Gamma (\alpha )}\int_{a}^{t}\psi ^{\prime }(s)(\psi(t)-\psi (s))^{\alpha -1}f(s,u(s),F_{u}(s))ds\bigg| \\
\leq& \bigg|\frac{u_{a}^{\ast }(\psi (t)-\psi (a))^{\gamma +\epsilon(\beta -1)-1}}{\Gamma (\gamma +\epsilon (\beta -1))}-\frac{u_{a}(\psi(t)-\psi (a))^{\gamma -1}}{\Gamma (\gamma )}\bigg| \\
&+\bigg|\int_{a}^{t}\psi ^{\prime }(s)\left[ \frac{(\psi (t)-\psi(s))^{\alpha -\epsilon -1}}{\Gamma (\alpha -\epsilon )}-\frac{(\psi (t)-\psi(s))^{\alpha -\epsilon -1}}
{\Gamma (\alpha )}\right] f(s,u^{\ast}(s),F_{u^{\ast }}(s))ds\\
&+\frac{1}{\Gamma (\alpha )}\int_{a}^{t}\psi ^{\prime }(s)(\psi (t)-\psi(s))^{\alpha -\epsilon -1}\left[ f(s,u^{\ast }(s),F_{u^{\ast}}(s))-f(s,u(s),F_{u}(s))\right] ds \\
&+\int_{a}^{t}\psi ^{\prime }(s)\left[ \frac{(\psi (t)-\psi(s))^{\alpha -\epsilon -1}}{\Gamma (\alpha )}-\frac{(\psi (t)-\psi(s))^{\alpha -1}}{\Gamma (\alpha )}\right] f(s,u(s),F_{u}(s))ds\bigg|.
\end{align*}%
Since
\begin{align*}
| F_{u^{\ast }}(t)-F_{u}(t)|=&| f(t,u^{\ast}(t),F_{u^{\ast }}(t))-f(t,u(t),F_{u}(t))| \\
&\leq M|u^{\ast }(t)-u(t)|+M^{\ast }|F_{u^{\ast }}(t)-F_{u}(t)| \\
&\leq\frac{M}{1-M^{\ast }}| u^{\ast }(t)-u(t)|.
\end{align*}
Then%
\begin{eqnarray*}
\left\vert u^{\ast }(t)-u(t)\right\vert &\leq &\left\vert \frac{u_{a}^{\ast
}(\psi (t)-\psi (a))^{\gamma +\epsilon (\beta -1)-1}}{\Gamma (\gamma
+\epsilon (\beta -1))}-\frac{u_{a}(\psi (t)-\psi (a))^{\gamma -1}}{\Gamma
(\gamma )}\right\vert \\
&&+\left\Vert f\right\Vert \left\vert \frac{(\psi (t)-\psi (a))^{\alpha
-\epsilon }}{\Gamma (\alpha -\epsilon +1)}-\frac{(\psi (t)-\psi (a))^{\alpha
-\epsilon }}{\Gamma (\alpha )\Gamma (\alpha -\epsilon )}\right\vert \\
&&+\frac{M}{1-M^{\ast }}\frac{1}{\Gamma (\alpha )}\int_{a}^{t}\psi ^{\prime
}(s)(\psi (t)-\psi (s))^{\alpha -\epsilon -1}\left\vert u^{\ast
}(s)-u(s)\right\vert ds \\
&&+\left\Vert f\right\Vert \left[ \frac{(\psi (t)-\psi (s))^{\alpha
-\epsilon }}{\Gamma (\alpha )\Gamma (\alpha -\epsilon )}-\frac{(\psi
(t)-\psi (s))^{\alpha }}{\Gamma (\alpha +1)}\right] \\
&=&A(t)+\frac{M}{1-M^{\ast }}\frac{1}{\Gamma (\alpha )}\int_{a}^{t}\psi
^{\prime }(s)(\psi (t)-\psi (s))^{\alpha -\epsilon -1}\left\vert u^{\ast
}(s)-u(s)\right\vert ds,
\end{eqnarray*}%
where $A(t)$ is defined as in \eqref{t4}.By applying Lemma \ref{L2}, we
conclude that%
\begin{equation*}
|u^{\ast }(t)-u(t)|\leq A(t)+\int_{a}^{t}\left[ \sum_{k=1}^{\infty }\left( \frac{M\Gamma
(\alpha -\epsilon )}{\Gamma (\alpha )(1-M^{\ast })}\right) ^{k}\frac{\psi^{\prime }(s)(\psi (t)-\psi (s))^{k(\alpha -\epsilon )-1}}{\Gamma (k(\alpha-\epsilon ))}A(s)\right] ds.
\end{equation*}
\end{proof}
Next, we consider the following fractional differential equation
\begin{equation}
D_{a^{+}}^{\alpha ,\beta ;\psi }u(t)=f(t,u(t),D_{a^{+}}^{\alpha ,\beta ;\psi
}u(t)),\text{ \ }0<\alpha <1,0\leq \beta \leq 1,t>a.  \label{u1}
\end{equation}
with initial condition
\begin{equation}
I_{a^{+}}^{1-\gamma ;\psi }u(a)=u_{a}+\delta .\qquad \qquad \ \ \qquad \
\gamma =\alpha +\beta -\alpha \beta ,  \label{u2}
\end{equation}
\begin{theorem}
Assume that hypotheses of Theorem \ref{BB} hold. Let $u$ and $u^{\ast }$ are
solutions of the Cauchy-type problems \eqref{11}-\eqref{12} and \eqref{u1}-\eqref{u2} respectively. Then
\begin{equation}
\left\vert u(t)-u^{\ast }(t)\right\vert \leq \left\vert \delta \right\vert
(\psi (t)-\psi (a))^{\gamma -1}E_{\alpha ,\gamma }\left[ \frac{M}{(1-M^{\ast
})}(\psi (t)-\psi (a))^{\alpha }\right] ,\text{ }t\in \lbrack a,b].  \notag
\end{equation}
\end{theorem}
\begin{proof}
In view of Theorem \ref{BB}, we have $u(t)=\underset{k\rightarrow \infty }{%
\lim }u_{k}(t)$ with%
\begin{equation}
u_{0}(t)=\frac{u_{a}}{\Gamma (\gamma )}(\psi (t)-\psi (a))^{\gamma -1}
\label{u3}
\end{equation}%
and%
\begin{align}\label{u5}
u_{k}(t)&=u_{0}(t)+I_{a^{+}}^{\alpha ;\psi }F_{u_{k-1}}(t) \nonumber \\
&=\frac{u_{a}}{\Gamma (\gamma )}(\psi (t)-\psi (a))^{\gamma -1}+\frac{1}{\Gamma (\alpha )}\int_{a}^{t}\psi ^{\prime }(s)(\psi (t)-\psi(s))^{\alpha -1}f(s,u_{k-1}(s),F_{u_{k-1}}(s))ds.
\end{align}
Clearly, we can write $u^{\ast }(t)=\underset{k\rightarrow \infty }{\lim }%
u_{k}^{\ast }(t)$ with
\begin{equation}
u_{0}^{\ast }(t)=\frac{(u_{a}+\delta )}{\Gamma (\gamma )}(\psi (t)-\psi
(a))^{\gamma -1},  \label{u4}
\end{equation}%
and%
\begin{align}\label{u7}
u_{k}^{\ast}(t)&=u_{0}^{\ast }(t)+I_{a^{+}}^{\alpha ;\psi
}F_{u_{k-1}^{\ast }}(t)  \notag \\
&=\frac{(u_{a}+\delta )}{\Gamma (\gamma )}(\psi (t)-\psi (a))^{\gamma -1}+\frac{1}{\Gamma (\alpha )}\int_{a}^{t}\psi ^{\prime }(s)(\psi (t)-\psi
(s))^{\alpha -1}f(s,u_{k-1}^{\ast }(s),F_{u_{k-1}^{\ast }}(s))ds.
\end{align}
By \eqref{u3} and \eqref{u4} we get%
\begin{equation}\label{u6}
\left\vert u_{0}(t)-u_{0}^{\ast }(t)\right\vert =\left\vert \frac{u_{a}}{\Gamma (\gamma )}(\psi (t)-\psi (a))^{\gamma -1}-\frac{(u_{a}+\delta )}{\Gamma (\gamma )}(\psi (t)-\psi (a))^{\gamma -1}\right\vert
\leq \left\vert \delta \right\vert \frac{(\psi (t)-\psi (a))^{\gamma -1}}{\Gamma (\gamma )}.
\end{equation}%
Using relations (\ref{u3}), (\ref{u5}), (\ref{u4}), (\ref{u7}), the
Lipschitz condition (\ref{z}) and the inequality (\ref{u6}), we get%
\begin{align*}
|u_{1}(t)-u_{1}^{\ast }(t)|\leq&|\delta| \frac{(\psi (t)-\psi (a))^{\gamma -1}}{\Gamma (\gamma )} \\
&+\frac{1}{\Gamma (\alpha )}\int_{a}^{t}\psi ^{\prime }(s)(\psi (t)-\psi(s))^{\alpha -1}|f(s,u_{0}(s),F_{u_{0}}(s))-f(s,u_{0}^{\ast}(s),F_{u_{0}^{\ast }}(s)|ds \\
\leq& |\delta| \frac{(\psi (t)-\psi (a))^{\gamma -1}}{\Gamma (\gamma )}+\frac{1}{\Gamma (\alpha )}\int_{a}^{t}\psi ^{\prime}(s)(\psi (t)-\psi (s))^{\alpha -1}M| u_{0}(t)-u_{0}^{\ast}(t)| ds \\
&+\frac{1}{\Gamma (\alpha )}\int_{a}^{t}\psi ^{\prime }(s)(\psi (t)-\psi(s))^{\alpha -1}M^{\ast }| F_{u_{0}}(s)-F_{u_{0}^{\ast}}(s)| ds \\
\leq& | \delta | \frac{(\psi (t)-\psi (a))^{\gamma -1}}{\Gamma (\gamma )}+\frac{M}{1-M^{\ast }}\frac{1}{\Gamma (\alpha )}\int_{a}^{t}\psi ^{\prime }(s)(\psi (t)-\psi (s))^{\alpha -1}|u_{0}(t)-u_{0}^{\ast }(t)| ds \\
\leq& |\delta | \frac{(\psi (t)-\psi (a))^{\gamma -1}}{\Gamma (\gamma )}+\frac{M|\delta |}{1-M^{\ast }}\frac{(\psi (t)-\psi (a))^{\gamma -1}}{\Gamma (\gamma )}\frac{1}{\Gamma (\alpha )}%
\int_{a}^{t}\psi ^{\prime }(s)(\psi (t)-\psi (s))^{\alpha -1}ds \\
=&|\delta| \frac{(\psi (t)-\psi (a))^{\gamma -1}}{\Gamma (\gamma )}+\frac{M|\delta|}{1-M^{\ast }}\frac{(\psi (t)-\psi (a))^{\gamma +\alpha -1}}{\Gamma (\gamma +\alpha )}.
\end{align*}
Hence,
\begin{equation}
\left\vert u_{1}(t)-u_{1}^{\ast }(t)\right\vert \leq \left\vert \delta
\right\vert (\psi (t)-\psi (a))^{\gamma -1}\sum_{i=0}^{1}\left( \frac{M}{%
1-M^{\ast }}\right) ^{i}\frac{(\psi (t)-\psi (a))^{\alpha i}}{\Gamma (\gamma
+\alpha i)}.  \label{u8}
\end{equation}%
On the other hand, we have%
\begin{align*}
|u_{2}(t)-u_{2}^{\ast}(t)|\leq&|\delta|\frac{(\psi (t)-\psi (a))^{\gamma -1}}{\Gamma (\gamma )} \\
&+\frac{1}{\Gamma (\alpha )}\int_{a}^{t}\psi ^{\prime }(s)(\psi (t)-\psi(s))^{\alpha -1}|f(s,u_{1}(s),F_{u_{1}}(s))-f(s,u_{1}^{\ast}(s),F_{u_{1}^{\ast }}(s)| ds \\
\leq&|\delta|\frac{(\psi (t)-\psi (a))^{\gamma -1}}{\Gamma (\gamma )}+\frac{1}{\Gamma (\alpha )}\int_{a}^{t}\psi ^{\prime}(s)(\psi (t)-\psi (s))^{\alpha -1}M|u_{1}(t)-u_{1}^{\ast}(t)|ds \\
&+\frac{1}{\Gamma (\alpha )}\int_{a}^{t}\psi ^{\prime }(s)(\psi (t)-\psi(s))^{\alpha -1}M^{\ast }|F_{u_{1}}(s)-F_{u_{1}^{\ast}}(s)| ds \\
\leq &|\delta |\frac{(\psi (t)-\psi (a))^{\gamma -1}}{\Gamma (\gamma )}+\frac{M}{1-M^{\ast }}\frac{1}{\Gamma (\alpha )}\int_{a}^{t}\psi ^{\prime }(s)(\psi (t)-\psi (s))^{\alpha -1}|u_{1}(t)-u_{1}^{\ast }(t)| ds \\
\leq &|\delta | \frac{(\psi (t)-\psi (a))^{\gamma -1}}{\Gamma (\gamma )}+\frac{M|\delta |}{1-M^{\ast}}\frac{1}{\Gamma (\alpha )}\int_{a}^{t}\psi ^{\prime }(s)(\psi (t)-\psi (s))^{\alpha-1}(\psi(t)-\psi(a))^{\gamma-1} \\
&\times \sum_{i=0}^{1}\left( \frac{M}{1-M^{\ast }}\right)^{i}\frac{(\psi(t)-\psi (a))^{\alpha i}}{\Gamma (\gamma +\alpha i)}ds \\
\leq &|\delta |\frac{(\psi (t)-\psi (a))^{\gamma -1}}{\Gamma (\gamma )}+\frac{M|\delta|}{1-M^{\ast }}\frac{1}{\Gamma (\alpha )}\sum_{i=0}^{1}\left( \frac{M}{1-M^{\ast }}\right) ^{i}\frac{(\psi (t)-\psi (a))^{\alpha i}}{\Gamma (\gamma +\alpha i)} \\
&\times \left( \int_{a}^{t}\psi ^{\prime }(s)(\psi (t)-\psi (s))^{\alpha-1}(\psi (t)-\psi (a))^{\gamma +\alpha i-1}ds\right) \\
\leq &|\delta | \frac{(\psi (t)-\psi (a))^{\gamma -1}}{\Gamma (\gamma )}+\frac{M|\delta|}{1-M^{\ast }}\frac{%
(\psi (t)-\psi (a))^{\gamma -1}}{\Gamma (\alpha )\Gamma (\gamma )}%
\int_{a}^{t}\psi ^{\prime }(s)(\psi (t)-\psi (s))^{\alpha -1}ds \\
&+\left( \frac{M}{1-M^{\ast }}\right) ^{2}|\delta|
\frac{(\psi (t)-\psi (a))^{\gamma +\alpha -1}}{\Gamma (\alpha )\Gamma
(\gamma +\alpha )}\int_{a}^{t}\psi ^{\prime }(s)(\psi (t)-\psi (s))^{\alpha
-1}ds \\
\leq &|\delta| \frac{(\psi (t)-\psi (a))^{\gamma -1}}{%
\Gamma (\gamma )}+\frac{M|\delta|}{1-M^{\ast }}\frac{%
(\psi (t)-\psi (a))^{\gamma -1}(\psi (t)-\psi (a))^{\alpha }}{\Gamma (\gamma
+\alpha )} \\
&+\left( \frac{M}{1-M^{\ast }}\right) ^{2}|\delta|\frac{(\psi (t)-\psi (a))^{\gamma -1}(\psi (t)-\psi (a))^{2\alpha }}{\Gamma
(\gamma +2\alpha )} \\
=&|\delta |(\psi (t)-\psi (a))^{\gamma
-1}\sum_{i=0}^{2}\left( \frac{M}{1-M^{\ast }}\right) ^{i}\frac{(\psi
(t)-\psi (a))^{\alpha i}}{\Gamma (\gamma +\alpha i)}.
\end{align*}
Using the mathematical induction, we get%
\begin{equation}
\left\vert u_{k}(t)-u_{k}^{\ast }(t)\right\vert \leq \left\vert \delta
\right\vert (\psi (t)-\psi (a))^{\gamma -1}\sum_{i=0}^{k}\left( \frac{M}{%
1-M^{\ast }}\right) ^{i}\frac{(\psi (t)-\psi (a))^{\alpha i}}{\Gamma (\gamma
+\alpha i)}.  \label{u9}
\end{equation}%
Taking the limit $m\rightarrow \infty $ in inequation \eqref{u9}, we obtain
\begin{equation*}
\left\vert u_{k}(t)-u_{k}^{\ast }(t)\right\vert \leq \left\vert \delta
\right\vert (\psi (t)-\psi (a))^{\gamma -1}E_{\gamma ,\alpha }\left( \frac{M%
}{1-M^{\ast }}(\psi (t)-\psi (a))^{\alpha }\right) .
\end{equation*}
\end{proof}
\section{An example.}\label{7}

\label{nn} Fix a kernel function $\psi :[0,1]\rightarrow
\mathbb{R}
$ such that $\psi (t)=\sqrt{t+1}$. Let $\alpha =\frac{1}{2},\beta =\frac{1}{3%
}$ and $\gamma =\frac{2}{3}$. Consider the following implicit Cauchy-type problem:
\begin{equation}
D_{0^{+}}^{\frac{1}{2},\frac{1}{3},\psi }u(t)=f(t,u(t),D_{0^{+}}^{\frac{1}{2}%
,\frac{1}{3},\psi }u(t)),\text{ \ },t>0,  \label{3}
\end{equation}%
\begin{equation}
I_{0^{+}}^{\frac{1}{3},\psi }u(0)=u_{0}.\qquad u_0\in\mathbb{R},   \label{4}
\end{equation}%
where
\begin{equation*}
f(t,u(t),D_{0^{+}}^{\frac{1}{2},\frac{1}{3},\psi }u(t))=\frac{1}{(1+9e^{t})\left( 1+\left\vert u(t)\right\vert+\left\vert D_{0^{+}}^{\frac{1}{2},\frac{1}{3},\psi }u(t)\right\vert \right)}.
\end{equation*}

It is easy to see that $f(t,u(t),D_{0^{+}}^{\frac{1}{2},\frac{1}{3},\psi
}u(t)\in C_{1-\frac{2}{3},\psi }[0,1].$ Moreover, for any $x,y,x^{\ast
},y^{\ast }\in \mathcal{\mathbb{R^{+}}}$ and $t\in (0,1],$ we have
\begin{align*}
| f(t,x,y)-f(t,x^{\ast},&y^{\ast})|_{C_{1-\frac{2}{3},\psi }[0,1]}=\max_{t\in[0,1]}[\psi (t)-\psi (0)]^{\frac{1}{3}}|f(t,x,y)-f(t,x^{\ast },y^{\ast })| \\
&=\max_{t\in[0,1]}[\psi (t)-\psi (0)]^{\frac{1}{3}%
}\bigg|\frac{1}{(1+9e^{t})(1+| x| +|y|)}-\frac{1}{(1+9e^{t})(1+|x^{\ast }|+|y^{\ast }|)}\bigg| \\
&\leq \frac{1}{10}\max_{t\in[0,1]}[\psi (t)-\psi (0)]^{%
\frac{1}{3}}\bigg|\frac{| x^{\ast }|+|y^{\ast }| -| x| -|y|}{(1+| x| +|y|)(1+| x^{\ast}| +|y^{\ast }|)}\bigg| \\
&\leq \frac{1}{10}\max_{t\in[0,1]}[\psi (t)-\psi (0)]^{%
\frac{1}{3}}[|x-x^{\ast }| +|y-y^{\ast}|]\\
&=\frac{1}{10}|x-x^{\ast }|_{C_{1-\frac{2}{3},\psi
}[0,1]}+\frac{1}{10}|y-y^{\ast }|_{C_{1-\frac{2}{3}%
,\psi }[0,1]}.
\end{align*}
Therefore, all the conditions of Theorem \ref{BB} are satisfied with $%
M=M^{\ast }=\frac{1}{10}$. It is easy to check that the condition (\ref{t1})
holds, i.e.
\begin{equation*}
\left(\frac{\Gamma (\gamma )\left[ (\psi (t)-\psi (a)\right]^{\alpha }}{\Gamma (\gamma +\alpha )}\frac{M}{1-M^{\ast }}\right) =\left(\frac{\Gamma (\frac{2}{3})\left[\sqrt{(t+1}-1\right]^{\frac{1}{2}}}{\Gamma(\frac{2}{3}+\frac{1}{2})}\frac{\frac{1}{10}}{1-\frac{1}{10}}\right)<1
\end{equation*}
for all $t\in \lbrack 0,1].$ Now we can apply Theorem \ref{BB} and
conclude that Cauchy-type problem (\ref{3})-(\ref{4}) has a unique solution in $_{C_{%
\frac{1}{3},\sqrt{t+1}}[0,1]}.$

\end{document}